\documentclass[11pt,reqno]{amsart}

\usepackage{graphicx} 
\usepackage{xcolor}
\usepackage{amssymb}
\usepackage{amsmath}
\usepackage{amsthm}
\usepackage{hyperref} 
\usepackage{parskip}
\usepackage{soul}
\usepackage{multicol}
\usepackage{cancel}

\usepackage{pgfplots}
\pgfplotsset{compat=1.15}
\usepackage{mathrsfs}
\usetikzlibrary{arrows}
\usepackage{tikz-cd}

\newtheorem{theorem}{Theorem}[section]
\newtheorem*{theoremA}{Theorem A}
\newtheorem{proposition}[theorem]{Proposition}
\newtheorem{corollary}{Corollary}[theorem] 

\newtheorem{definition}[theorem]{Definition} 
\newtheorem{rmk}[theorem]{Remark}

\title{Uniqueness of capillary disks in three-dimensional domains}
\author{Henrique Nogueira Bastos}
\thanks{This study was financed in part by the Coordenação de Aperfeiçoamento de Pessoal de Nível Superior - Brasil (CAPES) - Finance Code 001.}
\address{IMPA, Rio de Janeiro RJ 22460-320 Brazil \newline \textit{Email address}: henrique.bastos@impa.br}

\usepackage[backend=biber,
style=numeric,
sorting=nyt]{biblatex}
\begin{document}

\begin{abstract}
We prove uniqueness results for capillary disks in three-dimensional domains that are modeled by an elliptic PDE, under the assumption that the domain admits a family of surfaces with suitable properties. Our main theorem generalizes Nitsche's result for capillary constant mean curvature disks in the Euclidean ball and is inspired by the extension of Hopf's uniqueness theorem for constant mean curvature spheres in Euclidean space due to Gálvez and Mira.
\end{abstract}

\maketitle

\section{Introduction}

 An immersed surface in a domain of a Riemannian three-manifold satisfies the capillary boundary condition if its boundary lies on the boundary of the domain and the surface meets the boundary of the domain at a constant angle along its own boundary. In the particular case where the constant angle is $\pi/2$, the capillary condition is also called the free boundary condition.
 
 In 1985, Nitsche \cite{Nitsche} proved that \textit{a constant mean curvature (CMC) disk satisfying the capillary boundary condition in a three-dimensional Euclidean ball is necessarily a totally umbilical disk}. In particular, it is either a planar disk or a spherical cap.

In \cite{Souam} and \cite{Ros-Souam-1997}, Ros and Souam established a version of Nitsche's Theorem for disks satisfying the capillary boundary condition with respect to geodesic balls in three-dimensional space forms. Later, Fraser and Schoen \cite{Fraser-Schoen-Uniqueness-Disks} proved that free boundary two-dimensional disks with parallel mean curvature, immersed in $n$-dimensional geodesic balls in space forms, are totally umbilical. (For a different higher dimensional generalization of Nitsche's Theorem, see the work of M. Li, G. Wang and L. Weng \cite{LiWanWen}, and also the work of Gaia \cite{Gai}).

The idea behind the proof of these results involves the use of Hopf differentials of CMC surfaces in space forms. They are holomorphic forms that vanish if and only if the surface is totally umbilical. The capillary or free boundary conditions are shown to be the appropriate conditions that guarantee the vanishing of the Hopf differential of every CMC disk that satisfies them. Originally, the holomorphic differential technique was used by H. Hopf \cite{Hopf1951} in the proof of his famous theorem, which states that \textit{any immersed CMC sphere in $\mathbb{R}^3$ is a round sphere}. 

Hopf's proof strategy has been generalized to three-manifolds other than the three-dimensional space forms by Abresch and Rosenberg \cite{AbrRos1}, \cite{AbrRos2}, in the case of homogeneous three-manifolds with a four-dimensional isometry group. The case of general homogeneous manifolds, for which the holomorphic differential argument seems to be out of reach, was dealt with in a series of works by several authors, culminating in the classification theorem of Meeks, Mira, Perez, and Ros \cite{MeeMirPerRos}. (See the survey article \cite{MirPer} for the full story of these developments).

However, even the new methods employed in \cite{MeeMirPerRos} could not treat situations in which the three-manifold is not homogeneous. Recently, Gálvez and Mira \cite{Galvez-Mira-Uniqueness} obtained a uniqueness theorem for immersed spheres, which are more generally modeled by an elliptic PDE, on manifolds that satisfy certain extra properties. Briefly speaking, the hypotheses of their result require the existence of a family of oriented surfaces, called a \textit{transitive family}, that foliates the Grassmannian of oriented 2-planes in the tangent space of the three-manifold, and whose members are modeled by the same elliptic PDE. The conclusion of their theorem is that any immersed sphere modeled by that PDE belongs to the transitive family. 

Their result is a wide-ranging generalization of Hopf's Theorem. In fact, while at first it was not clear if this class of three-manifolds included non-homogeneous spaces, the work of Ambrozio, Marques and Neves \cite{AmbMarNev} showed that there are plenty of non-homogeneous metrics on the three-sphere containing transitive families of embedded minimal two-spheres. In particular, Gálvez-Mira uniqueness theorem provides the full classification of immersed minimal two-spheres in these spaces.


In this paper, we follow the approach of Gálvez-Mira to uniqueness results for two-spheres in three-manifolds and establish new uniqueness results for capillary disks in three-dimensional domains, which are generalizations of Nitsche's results.

\subsection{Main result} The set-up is as follows. Consider a complete orientable Riemannian three-manifold $(M,g)$, and denote by $G^2(M)$ the Grassmannian of oriented $2$-planes in $M$. Let $\Sigma$ be a connected orientable surface (without boundary). Given an immersion $\iota : \Sigma \rightarrow M$, its associated Legendrian lift is the map $\mathcal{L}_\Sigma:\Sigma\rightarrow G^2(M)$ defined by $\mathcal{L}(p)=(\iota(p),T_{\iota(p)}\Sigma)$ for all $p\in \Sigma$. (If the context is not ambiguous, we may omit $\iota$ and regard $\Sigma$ as a subset of $M$).

Let $\Omega$ be a domain with smooth boundary in $M$. An immersed surface $\Sigma$ in $M$ satisfies the capillary boundary condition with respect to $\partial \Omega$ if its intersection with $\partial\Omega$ is transversal and meets $\partial\Omega$ at a constant angle $\alpha\in (0,\pi)$. 

Inspired by the properties of the family of CMC capillary disks in the Euclidean ball (which are subsets of ambient spheres or planes, depending on the value of the mean curvature, and intersect the boundary of the ball transversely), we propose a notion of transitive family that is well-adapted to the capillary boundary condition with respect to a given domain. (Compare with \cite{Galvez-Mira-Uniqueness}, Definition 2.1):

\begin{definition}[Transitive Capillary Family]\label{def.Transitive Capillary Family} Let $\mathcal{S}$ be a family of immersed oriented surfaces (without boundary) in $M$. We say that $\mathcal{S}$ is a transitive capillary family in $M$ with respect to the smooth domain $\Omega\subset M$ if it satisfies:

    \begin{enumerate}
        \item For every $S\in \mathcal{S}$, the Legendrian lift $\mathcal{L}_S$ is an embedding of $S$ into $G^2(M)$.
        \item\label{(ii)} For every $(p,\Pi_p)\in G^2(M)$ with $p\in \Omega$, there exists a unique surface $S=S(p,\Pi_p)\in \mathcal{S}$ such that $(p,\Pi_p)=\mathcal{L}_S(p)$.
        \item\label{(iii)} If $S\in\mathcal{S}$ intersects $\partial\Omega$ transversely at some point $p$, then $S$ meets $\partial\Omega$ at a constant contact angle along the connected component of $S\cap\partial\Omega$ that contains $p$.
        \item The family $\mathcal{S}=\{S(p,\Pi_p)\,|\,p\in\Omega\text{ and }(p,\Pi_p)\in G^2(M)\}$ is of class $C^3$ with respect to $(p,\Pi_p)$.
    \end{enumerate}    
\end{definition}

    In particular, by virtue of \eqref{(ii)} and \eqref{(iii)}, if $\mathcal{S}$ is a transitive capillary family with respect to $\Omega$, then for every $p\in \partial \Omega$ and every angle $\alpha\in (0,\pi)$, there exists a member $S$ of the family satisfying the capillary boundary condition with respect to $\partial \Omega$, meeting it at $p$ with the prescribed angle $\alpha$. Note also that $\mathcal{S}$ may contain surfaces that lie entirely inside $\Omega\setminus \partial \Omega$, or even that meet $\partial \Omega$ tangentially. 

   \indent Let $\Phi=\Phi(x,y,z,p,q,r,s,t)\in C^{1,\alpha}(\mathcal{U})$, where $\mathcal{U}\subset \mathbb{R}^{8}$ is an open convex set, assumed to intersect the vector subspace $\{x=y=z=0\}$. We say that the second-order PDE in two variables
\begin{equation}\label{int.Phi=0}
        \Phi(x,y,f,f_x,f_y,f_{xx},f_{xy},f_{yy})=0
\end{equation}
is elliptic if $\Phi_r>0$ and $4\Phi_r\Phi_t-(\Phi_s)^2>0$ hold on $\mathcal{U}$.

A function $h\in C^2(D)$, defined on a planar domain $D\subset\mathbb{R}^2$ containing the origin, is a solution to the equation $\Phi=0$ if $\{(u,v,h,h_u,h_v,h_{uu},h_{uv},h_{vv})\in \mathbb{R}^8\,|\,\ (u,v)\in D\}\subset\mathcal{U}$ and $h=h(u,v)$ satisfies \eqref{int.Phi=0}.

Let $(u,v,z)$ be local coordinates on $M$. We call them \textit{adapted coordinates} to $(p,\Pi)\in G^2(M)$ if, in these coordinates, $p=(0,0,0)$ and $\Pi=\text{span}\{\partial_u(0),\partial_v(0)\}$ with its usual orientation. 

An elliptic PDE field on $M$ consists of a covering of $M$ by coordinate charts adapted to each $(q,\Pi)\in G^2(M)$, on which an elliptic PDE $\Phi_{(q,\Pi)}=0$ as in \eqref{int.Phi=0} is defined.

Given a smooth compact domain $\Omega$ on $M$, we say that the elliptic PDE field on $M$ is adapted to $\Omega$ if, in addition, the following holds: for $(p,\Pi) \in G^2(M)$ with $p \in \Omega \setminus \partial \Omega$, the adapted coordinates parametrize a subset of $\Omega \setminus \partial \Omega$; while for $p \in \partial \Omega$ and $\Pi$ transverse to $T_p \partial \Omega$, the adapted coordinates are chosen so that points in $\Omega$ correspond to $\{u \geq 0\}$ and points in $\partial \Omega$ correspond to points in $\{u = 0\}$. These further restrictions will make it easier to express the boundary condition for the PDE $\Phi_{(p,\Pi)} = 0$, which is the manifestation of the capillary boundary condition.

\begin{definition}[Modeled Surface] An immersed orientable surface $\Sigma$ is said to be modeled by an elliptic PDE on $M$, adapted to $\Omega$, if for every $q \in \Sigma$ and adapted coordinates $(u,v,z)$ to $(q,\Pi)$ in which $\Sigma$ can be locally written as a vertical graph $z = h(u,v)$ of a $C^2$ function, the function $h$ is a solution to the elliptic PDE $\Phi_{(q,\Pi)} = 0$.
\end{definition}


Our main result concerns immersed surfaces $\Sigma$ that are modeled by an elliptic PDE on $(M,g)$, adapted to a smooth domain $\Omega$, and which intersect $\Omega$ in an immersed disk satisfying the capillary boundary condition with respect to $\Omega$.

\begin{theoremA}\label{int.theo.1}
Let $\Omega$ be a smooth domain in a complete orientable Riemannian three-manifold $(M,g)$. Let $\mathcal{A}$ be a class of immersed oriented surfaces in $M$ such that:
    \begin{itemize}
        \item $\mathcal{A}$ is modeled by an elliptic PDE (adapted to $\Omega$) on $M$;
        \item there exists a transitive capillary family $\mathcal{S}\subset\mathcal{A}$ in $M$ with respect to $\Omega$.
    \end{itemize}
 If an immersed connected oriented surface $\Sigma\in \mathcal{A}$ intersects $\Omega$ in an immersed disk satisfying the capillary boundary condition with respect to $\Omega$, then $\Sigma$ is contained in a surface $S\in \mathcal{S}$ that intersects $\partial \Omega$.
\end{theoremA}

In order to see how this theorem can be used to prove uniqueness results for disks satisfying the capillary boundary condition, let us show how Nitsche's Theorem can be recovered in the case of minimal surfaces. The set of totally geodesic planes in $\mathbb{R}^3$ that intersect the unit ball $\mathbb{B}^3$ centered at the origin transversely forms a transitive capillary family with respect to $\mathbb{B}^3$. Each element is modeled by an elliptic PDE corresponding to the minimal surface equation for graphs. Moreover, each element determines an embedded disk in $\mathbb{B}^3$. 

Let $\Sigma$ be an immersed minimal disk in $\mathbb{B}^3$ satisfying the capillary boundary condition. It is well known that $\Sigma$ admits a unique open minimal extension $\tilde{\Sigma}$ beyond $\mathbb{B}^3$ (since its boundary is analytic – see \cite{Lew} for the free boundary case and \cite{KinNirSpr}, Theorem 5.2, for the general case; one may extend it, for instance, using the solution to Bj\"orling's problem). Applying Theorem \ref{int.theo.1} to $\Sigma$, we conclude that $\Sigma=\tilde{\Sigma}\cap \mathbb{B}^3$ is contained in the intersection of a totally geodesic plane in $\mathbb{R}^3$ with $\mathbb{B}^3$. In particular, $\Sigma$ is an embedded totally geodesic disk.

\subsection{Strategy of the proof} We now describe the main ideas of the proof, which closely follow those developed by Gálvez and Mira in \cite{Galvez-Mira-Uniqueness}. For every $q\in\Sigma$, we define a bilinear form on $T_q\Sigma$ given by
\begin{equation}\label{int.2}
    \sigma_q:=II^{\Sigma}_{q}-II^{S(q,T_q\Sigma)}_q,
\end{equation}
where $II^\Sigma$ is the second fundamental form of $\Sigma$ in $M$, and $S(q,T_q\Sigma)$ denotes the unique element of the transitive capillary family passing through $q$ with $T_q\Sigma$ as its oriented tangent plane at $q$. By the definition of a transitive capillary family, this bilinear form is well-defined on $\Sigma$, of class $C^1$, and measures how far $\Sigma$ is from being contained in a member of the transitive family.

The assumption that $\Sigma$ and the transitive family are modeled by the same elliptic PDE is important for characterizing the zeros of $\sigma$, which are isolated unless the two surfaces locally coincide near $q$. Moreover, the local behavior near these interior zeros is described by the theory developed by L. Bers (See Theorem 6.12 in \cite{ColMin}, or more generally \cite{Bers1956}). Note that our set-up ensures that these points are always interior zeros of $\sigma$ in $\Sigma$, even when $q\in \partial \Omega\cap \Sigma$.

The idea of Gálvez and Mira, who were dealing with closed surfaces, was then to show that, if $\Sigma$  is not contained in the transitive family, then the fields of asymptotic directions of $\sigma$ have only isolated singularities of negative index. By the Poincaré-Hopf Theorem, this is incompatible with the topology of the sphere.

To extend this strategy to the case of a surface $\Sigma$ satisfying the capillary boundary condition with respect to a smooth domain $\Omega$, the main difficulty is to analyze the behavior of $\sigma$ near points of $\partial (\Omega \cap \Sigma)$. As we will prove, if $\Sigma$ is not contained in a member of the transitive capillary family, the capillary boundary condition implies that the fields of asymptotic lines of $\sigma$ are transversal to the boundary of $\Sigma \cap \Omega$ (see Theorem \ref{thm: mainA}). This provides sufficient control to extend these line fields to a sphere (the ``double'' of the disk $\Sigma \cap \Omega$), so that all singularities have negative index (see Theorem \ref{thm:negativeindex}). This leads to a contradiction, showing that $\Sigma$ must be contained in a member of the transitive capillary family with respect to $\Omega$.

We believe that the formulation of our main result could be made more natural and more general if some analytical facts were established regarding the regularity (and maybe the existence of extensions) of solutions to the general elliptic PDE \eqref{int.Phi=0} with capillary boundary conditions, and the local behavior near boundary points where two such solutions coincide. We will not attempt to pursue these directions here, focusing instead on the geometric aspects of the argument, which are neater in our setting and already lead to new applications, as discussed in the final section of the paper.

\textbf{Acknowledgements}: The author is supported by CAPES PhD Scholarship 88887.602445/2021-00. This work is part of the author’s PhD thesis at IMPA, carried out under the supervision of Professor Lucas Ambrozio. The author is deeply grateful to Professor Ambrozio for his guidance, encouragement, and numerous suggestions throughout this research. Thanks are also due to Carlos Toro and Diego Guajardo for fruitful discussions, to Zihui Zhao for pointing out reference \cite{KinNirSpr}, and to Pablo Mira for his comments and suggestions on an earlier version of this manuscript.

\section{Boundary behavior}

\subsection{Boundary conditions in adapted coordinates} Let $\Omega\subset M$ be a smooth domain with a smooth boundary in a three-dimensional Riemannian manifold $(M,g)$, and let $\Sigma$ be an oriented surface modeled by an elliptic PDE adapted to $\Omega$. For a point $p$ in $\partial(\Sigma\cap\Omega)$, let $h$ denote the function that locally represents $\Sigma$ near $p$ as a graph in some system of adapted coordinates. We compute the equation satisfied by $h$ at points corresponding to $\partial(\Sigma\cap \Omega)$, as a consequence of the capillary boundary condition.

\begin{proposition}\label{prop:condição de bordo}
    Let $\Sigma$ be an immersed capillary surface in $\Omega$ with contact angle $\alpha\in(0,\pi)$. Let $(u,v,z)$ be adapted coordinates  to $(p,T_p\Sigma)$ for some $p\in \partial(\Sigma\cap  \Omega)$, chosen so that $\{u=0\}$ corresponds to $\partial \Omega$ and $(0,v,z)$ are isothermal coordinates on $\partial \Omega$, with conformal factor $\lambda^2>0$. 
    
    If $\Sigma$ is locally parametrized as $X(u,v)=(u,v,h(u,v))$ near $p$, then
    \begin{equation}\label{eq:prop2.1}
     -g_{uu}(0,v,h(0,v))(1+h_v^2(0,v))\cos^2\alpha+\lambda^2(0,v,h(0,v)) h_u(0,v)^2\sin^2\alpha=0
    \end{equation}
    for all $v$ sufficiently close to $0$.
\end{proposition}
\begin{proof}
Let $g_{ij}$ denote the coefficients of the metric $g$ with respect to the adapted coordinates $(u,v,z)$; for instance, $g_{uu}=g(\partial_u,\partial_u)$. Note that $g_{uv}=g_{uz}=0$ along $u=0$, since $\partial_u$ is a normal vector field to $\partial \Omega$ at points $(0,v,z)$ by assumption.

Denote 
\begin{equation*}
    \begin{bmatrix}C_{vv} & C_{vz} \\ C_{zv} & C_{zz} \end{bmatrix}=\begin{bmatrix}g(X_v,X_v) & g(X_v,\partial_z) \\ g(\partial_z,X_v) & g(\partial_z,\partial_z) \end{bmatrix}=\begin{bmatrix}g_{vv}+h_v^2g_{zz} & h_vg_{zz} \\ h_vg_{zz} & g_{zz} \end{bmatrix},
\end{equation*}
and let $C^{ij}$ be the coefficients of the inverse matrix $[C_{ij}]^{-1}$; for instance, $C^{vv}=C_{zz}/\det C$.

Let $c$ be the curve given by $v\mapsto X(0,v)=(0,v,h(0,v))$ in $\partial \Omega$, which parametrizes a neighborhood of $p$ in $\partial(\Sigma\cap \Omega)$. The normal vector $\nu_c$ to $c$ in $\partial \Omega$ is given by
\begin{align*}
    \nu_c & = \frac{1}{\sqrt{C^{zz}}}(C^{zv}X_v+C^{zz}\partial_z) \\
    & = \frac{1}{\sqrt{C_{vv}\det C}}(-C_{zv}(\partial_v+h_v\partial_z)+C_{vv}\partial_z). \\
    & = \frac{1}{\sqrt{C_{vv}\det C}}(-h_vg_{zz}\partial_v+g_{vv}\partial_z).
\end{align*}

Denote by
\begin{equation*}
    \begin{bmatrix}G_{uu} & G_{uv} \\ G_{vu} & G_{vv} \end{bmatrix}=\begin{bmatrix}g(X_u,X_u) & g(X_u,X_v) \\ g(X_v,X_u) & g(X_v,X_v) \end{bmatrix} = \begin{bmatrix}g_{uu}+h_u^2g_{zz} & h_vh_{u}g_{zz} \\ h_vh_{u}g_{zz} & g_{vv}+h_v^2g_{zz} \end{bmatrix}
\end{equation*}
the coefficients of the induced metric on $\Sigma$ along $c$, and let $G^{ij}$ be the coefficients of the inverse matrix $[G_{ij}]^{-1}$. The conormal vector of $\Sigma$ along $c$ is then given by
\begin{align*}
    \nu_\Sigma & = \frac{1}{\sqrt{G^{uu}}}(G^{uu}X_u+G^{uv}X_v) \\
     & = \frac{1}{\sqrt{G_{vv}\det G}}(G_{vv}(\partial_u+h_u\partial_z)-G_{uv}(\partial_v+h_v\partial_z)) \\
     & = \frac{1}{\sqrt{G_{vv}\det G}}((g_{vv}+h_v^2g_{zz})\partial_u-h_uh_vg_{zz}\partial_v+h_ug_{vv}\partial_z).
\end{align*}

By the capillary condition with angle $\alpha\in (0,\pi)$, we have
\begin{align*}
    \cos^2\alpha  & = \langle \nu_c,\nu_\Sigma\rangle^2 \\ & = \frac{1}{G_{vv}C_{vv}\det G\det C}(h_uh_v^2g_{zz}^2g_{vv}+h_ug_{vv}^2g_{zz})^2 \\
    & = \frac{h_u^2\lambda^{12}(h_v^2+1)^2}{(g_{vv}+h_v^2g_{zz})^2(g_{uu}g_{vv}+h_u^2g_{vv}g_{zz}+h_v^2g_{uu}g_{zz})g_{vv}g_{zz}} \\
    & = \frac{h_u^2\lambda^2}{g_{uu}+h_u^2\lambda^2+h_v^2g_{uu}},  
\end{align*}
here we use that $g_{vv}=g_{zz}=\lambda^2$ on $u=0$. Therefore, the equation along the curve $c$ becomes
\begin{equation*}
    -g_{uu}(0,v,h(0,v))(1+h_v^2(0,v))\cos^2\alpha+\lambda^2(0,v,h(0,v)) h_u(0,v)^2\sin^2\alpha=0.
\end{equation*}
\end{proof}

\begin{rmk}
Suppose that $h$ and $\tilde{h}$ both describe surfaces as in Proposition \ref{prop:condição de bordo}, modeled by the same elliptic PDE. Then the difference $h-\tilde{h}$ satisfies a linear elliptic PDE on their common domain. Moreover, on points with $u=0$, it follows from \eqref{eq:prop2.1} that
\begin{align*}
    0 = & (h-\widetilde{h})_u\sin\alpha\int_0^1\lambda^\tau d\tau - (h-\widetilde{h})_v\cos\alpha\int_0^1\frac{h_v^\tau}{\sqrt{1+(h_v^{\tau})^2}} d\tau \\
    & +(h-\widetilde{h})\int_0^1\left(\lambda_z^\tau h_u^\tau\sin\alpha-\frac{1}{2}\frac{g^\tau_{uu,z}}{\sqrt{g^\tau_{uu}}}\sqrt{1+(h_v^\tau)^2}\right) d\tau,
\end{align*}
where $h^\tau(0,v)=\tau h(0,v)+(1-\tau)\widetilde{h}(0,v),\quad g_{uu}^\tau(0,v)=g_{uu}(0,v,h^\tau(0,v))$ and $\lambda^\tau(0,v)= \lambda(0,v,h^{\tau}(0,v))$. 

If $\alpha\neq \pi/2$, this defines an oblique boundary condition for $h-\tilde{h}$ (see \cite{gilbarg2001elliptic}, Example (v) in Chapter 10 and Section 17.9). If $\alpha=\pi/2$, it follows from \eqref{eq:prop2.1} that the free boundary condition reduces to a Neumann boundary condition for $h-\tilde{h}$.

\end{rmk}






\subsection{Asymptotic lines on the boundary}

Let $\Sigma$ be a connected surface modeled by an elliptic PDE in adapted coordinates, locally described as a graph $(u, v, h(u, v))$ around a point $p=(0,0,0)$. Let $\widetilde{\Sigma}$ be a surface in a transitive family, modeled by the same elliptic PDE field, that is tangent to $\Sigma$ at $p$, with graphical representation $(u,v,\widetilde{h}(u,v))$. 

If $\Sigma$ is not contained in $\widetilde{\Sigma}$, perhaps after an affine change of coordinates in $(u,v)$, the bilinear form $\sigma$ defined in \eqref{int.2} can be written near $p$ as
\begin{equation}\label{formula sigma}
\sigma=g(\partial_z(0),N(0)) \begin{bmatrix}
    w_{uu} & w_{uv} \\ w_{vu} & w_{vv}
\end{bmatrix}+o((\sqrt{u^2+v^2})^{k-2}),
\end{equation}
where 
\begin{equation*}
    h(u,v)-\widetilde{h}(u,v)=w(u,v)+o((\sqrt{u^2+v^2})^{k}),
\end{equation*} 
and $w(u,v)$ is a nonzero homogeneous harmonic polynomial $w(u,v)$ of degree $k\geq 2$. Here, $o(f(u,v))$ denotes a function $g(u,v)$ such that $g(u,v)/f(u,v)\rightarrow 0$ as $(u,v)\rightarrow (0,0)$. In particular, $\sigma$ has isolated singularities.

The capillary boundary condition determines how the principal line fields of $\sigma$ intersect the boundary of the compact surface $\Sigma\cap \Omega$.
\begin{theorem}\label{thm: mainA}
     Outside its singular set, the principal lines of $\sigma$ on $\Sigma$ are either tangential to $\partial(\Sigma\cap \Omega)$ or intersect it orthogonally.
\end{theorem}
\begin{proof} To study the behavior of the line fields at a point $p\in \partial (\Sigma\cap\Omega)$, we choose adapted coordinates $(u,v,z)$ around $p$ as the one chosen in Proposition \ref{prop:condição de bordo}.


We consider the parametrizations $X(u,v)=(u,v,h(u,v))$ for $\Sigma$ and $\widetilde{X}(u,v)=(u,v,\widetilde{h}(u,v))$ for the comparison surface $S(p,T_p\Sigma)=:\widetilde{\Sigma}$ in the transitive family, with $p=(0,0,0)$ and $h_v(0,0)=\widetilde{h}_v(0,0)=0$.

We first analyze the free boundary case, that is, when the contact angle $\alpha=\pi/2$. Substituting $\alpha=\pi/2$ into \eqref{eq:prop2.1}, we obtain $h_u(0,v)=0$, and hence $h_{uv}(0,0)=0$. 
Taking the difference with the analogous identity for $\widetilde{\Sigma}$, we obtain
\begin{equation*}
    w_{uv}(0,0)=(h-\widetilde{h})_{uv}(0,0)=0
\end{equation*}
where $w$ is the homogeneous harmonic polynomial appearing in \eqref{formula sigma}. Therefore, the directions $X_u(0,0)=\partial_u(p)$ and $X_v(0,0)=\partial_v(p)$ are principal directions of $\sigma$. In particular, $\partial_v(p)$ is not an asymptotic direction of $\sigma$.  The Theorem is proven in this case.

We now turn to the non free boundary case, that is, constant contact angle $\alpha\neq \pi/2$. Let $\nu_{\Sigma\cap \Omega}$ denote the conormal vector of $\Sigma\cap\Omega$, and $N_{\partial\Omega}$ denote the normal vector to $\partial\Omega$. Keeping the same notation from the proof of Proposition \ref{prop:condição de bordo}, we write
    \begin{align*}
        \nu_{\Sigma\cap \Omega}& =\frac{1}{\sqrt{G^{uu}}}(G^{uu}X_u+G^{uv}X_v) \\
        N_{\partial \Omega}& =\frac{1}{\sqrt{B^{uu}}}(B^{uu}X_u+B^{uv}X_v+B^{uz}\partial_z),
    \end{align*}
where $B_{ij}=g(X_i,X_j)$. (Here we abuse notation by writing $X_z=\partial_z$).

The capillary condition implies that the vectors $\nu(t)=\nu_{\Sigma\cap\Omega}(c(t))$ and $N(t)=N_{\partial \Omega}(c(t))$ satisfy
\begin{align*}
    \cos(\alpha) = & g(\nu(t),N(t)) \\
    = & \frac{1}{\sqrt{G^{uu}}}[G^{uu}g(X_u,N)+G^{uv}g(X_v,N)] \\
    = & \sqrt{G^{uu}}g(X_u,N)= \frac{\sqrt{G^{uu}}}{\sqrt{B^{uu}}}.
\end{align*}
Differentiating with respect to $t$, we obtain
\begin{align*}
    0 & = \frac{\sqrt{B^{uu}}}{2\sqrt{G^{uu}}}
    \ \frac{(G^{uu})_{,v}B^{uu}-(B^{uu})_{,v}G^{uu}}{(B^{uu})^2} \\
    & = \frac{\sqrt{B^{uu}}}{2\sqrt{G^{uu}}}\ \frac{B^{ui}B^{uj}B_{ij,v}G^{uu}-G^{ui}G^{uj}G_{ij,v}B^{uu}}{(B^{uu})^2},
\end{align*}
and hence
\begin{equation}\label{eq:derivada ao longo de c}
B^{ui}B^{uj}B_{ij,v}G^{uu}-G^{ui}G^{uj}G_{ij,v}B^{uu} = 0.
\end{equation}
Recall that $\partial_u$ is normal to $\partial\Omega$ and that $(0,v,z)$ are isothermal coordinates on $\partial\Omega$ in our choice of adapted coordinates. Also, $h_v(0,0)=\widetilde{h}_v(0,0)=0$. Since
\begin{equation*}    
    G(p)=\begin{bmatrix}
            G_{uu} & 0 \\
            0 & G_{vv}
        \end{bmatrix} \text{ and }
    B(p)=\begin{bmatrix}
            G_{uu} & 0 & B_{uz} \\
            0 & G_{vv} & 0 \\
            B_{uz} & 0 & B_{zz}
        \end{bmatrix},
\end{equation*}
and the inverse of this matrix at $p$ 
\begin{equation*}
    B(p)^{-1}=\frac{1}{G_{uu}B_{zz}-B_{uz}^2}\begin{bmatrix}
            B_{zz} & 0 & -B_{uz} \\
            0 & (G_{uu}B_{zz}-B_{uz}^2)/G_{vv} & 0 \\
            -B_{zu} & 0 & B_{uu}
        \end{bmatrix},
\end{equation*}
then, at $p$, equation \eqref{eq:derivada ao longo de c} yields
\begin{align*}
    0 = & (B^{uu})^2B_{uu,v}+2B^{uu}B^{uz}B_{uz,v}+(B^{uz})^2B_{zz,v}-G^{uu}B^{uu}G_{uu,v} \\
    = & B^{uu}G_{uu,v}(B^{uu}-G^{uu})+B^{uz}(2B^{uu}B_{uz,v}+B^{uz}B_{zz,v}) \\
    = & B^{uu}G_{uu,v}\left(\frac{B_{zz}}{G_{uu}B_{zz}-B_{uz}^2}-\frac{1}{G_{uu}}\right)+B^{uz}(2B^{uu}B_{uz,v}+B^{uz}B_{zz,v}) \\
    = & B^{uu}G_{uu,v}\frac{B_{uz}^2}{(G_{uu}B_{zz}-B_{uz}^2)G_{uu}}+B^{uz}(2B^{uu}B_{uz,v}+B^{uz}B_{zz,v}) \\
    = & -B^{uu}G_{uu,v}\frac{B_{uz}B^{uz}}{G_{uu}}+B^{uz}(2B^{uu}B_{uz,v}+B^{uz}B_{zz,v}).
\end{align*}
Taking the difference between the above expressions evaluated on $\Sigma$ and $\widetilde{\Sigma}$ at $p$, we obtain
\begin{align}\label{eq:auxiliar}
    0
    = & -B^{uu}\frac{B_{uz}B^{uz}}{G_{uu}}(G_{uu}-\widetilde{G}_{uu})_{,v} + 2B^{uu}B^{uz}(B_{uz}-\widetilde{B}_{uz})_{,v},
\end{align}
where we used that $B_{zz}=g(\partial_z,\partial_z)\circ c(t)$, $\widetilde{B}_{zz}=g(\partial_z,\partial_z)\circ \widetilde{c}(t)$, so that $B_{zz,v}=\widetilde{B}_{zz,v}$ at $t=0$, since $c(0)=\widetilde{c}(0)=p$ and $c'(0)=\widetilde{c}'(0)=\partial_v(p)$. 

Recall that $h_u(0,0)-\widetilde{h}_u(0,0)=0$, since the tangent planes of $\Sigma$ and $\widetilde{\Sigma}$ coincide at $p=(0,0,0)$. Since
\begin{equation*}
    G_{uu}=g_{uu}+2h_u g_{uz} + h^2_u g_{zz} \text{ and } \widetilde{G}_{uu}=g_{uu}+2\widetilde{h}_u g_{uz} + \widetilde{h}^2_u g_{zz},
\end{equation*}
and similarly
\begin{equation*}
    B_{uz}=g_{uz}+h_u g_{zz} \text{ and } \widetilde{B}_{uz}=g_{uz}+\widetilde{h}_u g_{zz},
\end{equation*}
differentiating with respect to $v$ at $p$ and substituting into \eqref{eq:auxiliar}, we obtain
\begin{align*}
    0= & -2B_{uz}B^{uz}(h-\widetilde{h})_{uv}B_{uz}+2B^{uz}G_{uu}(h-\widetilde{h})_{uv} B_{zz} \\
    = & 2B^{uz}(h-\widetilde{h})_{uv}(B_{zz}G_{uu}-B_{uz}^2).
\end{align*}
Note that $B_{zz}G_{uu}-B^2_{uz}\neq 0$. Therefore, at $(0,0)$,
\begin{equation*}
    (h-\widetilde{h})_{uv}=0 \text{ or } B^{uz}=0. 
\end{equation*}

Since we are now analyzing the case of constant angle $\alpha\neq \pi/2$, the second possibility cannot occur. Thus
\begin{equation*}
    (h-\widetilde{h})_{uv}(0,0)=0 
\end{equation*}
and the conclusion follows as in the free boundary case.    
\end{proof}

\section{Index of singularities and proof of the main result}





\subsection{Principal lines near singularities}


Let $\Sigma$ be a capillary surface modeled by the same elliptic PDE that defines the transitive family. Assume that $\Sigma$ is not contained in any element of the transitive family. Since $\Sigma$ is oriented and $\sigma$ defines a Lorentzian metric on $\Sigma$, except at isolated singularities, it follows that $\sigma$ induces two asymptotic line fields $L_1$ and $L_2$ on $\Sigma$, each with isolated singularities.

We say that a line field in a region is \emph{regular} if it is integrable, that is, if for every point there exists a $C^1$ regular curve $\gamma$, defined in a neighborhood of $0$, such that  $\gamma'(t)$ belongs to the line field at $\gamma(t)$ for all $t$. (See \cite{Hopf1946}, Chapter III, Section 1.1).

The next proposition is a consequence of Proposition \ref{thm: mainA} and shows that $\sigma$ induces a regular line field on the double surface $S:=\widehat{\Sigma}\cup_{id}\widehat{\Sigma}$, obtained by gluing two copies of $\widehat{\Sigma}:=\Sigma\cap \Omega$ along their boundaries via the identity map. To distinguish the two copies of $\widehat{\Sigma}$ in the double surface, we denote them by $\widehat{\Sigma}_i,\ i=1,2$, and use corresponding subscripts for objects defined on each copy. 

\begin{proposition}\label{prop:hipoteseshopf}
    Let $L$ be the line field on $S:=\widehat{\Sigma}_1\cup_{id}\widehat{\Sigma}_2$ defined by
    \begin{equation}
        L(x)=\begin{cases}
            L_1(x), & x\in \widehat{\Sigma}_1 \\ L_2(x), & x\in \widehat{\Sigma}_2\setminus \widehat{\Sigma}_1.
        \end{cases}
    \end{equation}
    Then $L$ is a regular line field on $S$, possibly with isolated singularities.
\end{proposition}
\begin{proof}
    It suffices to analyze nonsingular points $p$ of $\sigma$ lying in $\partial\widehat{\Sigma}$, that is, points along the gluing subset. 
    
    In a neighborhood $V=\mathbb{S}^1\times [0,1)$ of the connected component of $\partial\widehat{\Sigma}$ containing $p$, consider isothermal coordinates $(\theta,t)$ such that $p=(0,0)$. Let $\alpha:[0,\varepsilon)\rightarrow V_1$ and $\beta:[0,\varepsilon)\rightarrow V_2$ be curves with $\alpha(0)=\beta(0)=p$, satisfying $\alpha'(t)\in L_1(\alpha(t))$, $\beta'(t)\in L_2(\beta(t))$, and $|\alpha'(t)|=|\beta'(t)|=1$ for all $t\in[0,\varepsilon)$.

    Since asymptotic directions are invariant under reflection across principal lines, Theorem \ref{thm: mainA} implies that $\alpha'(0)=b\partial_\theta+c\partial_t$ and $\beta'(0)=-b\partial_\theta+c\partial_t$, for some constant $b,c\neq 0$. 
    
    The attaching map $\varphi: V_1\sqcup  V_2\rightarrow \mathbb{S}^1\times(-1,1)$ is given by    
    \begin{equation*}
        \varphi(x)=\begin{cases}
            (\theta,t), & x=(\theta,t)\in V_1, \\
            (\theta,-t), & x=(\theta,t)\in V_2.
        \end{cases}
    \end{equation*}
    Now define $\gamma:(-\varepsilon,\varepsilon)\rightarrow S$ by
    \begin{equation*}
        \gamma(s)=\begin{cases}
            \varphi\circ\alpha(s),& s\geq 0, \\ \varphi\circ\beta(-s),& s<0.
        \end{cases}
    \end{equation*}
    Since $\gamma_-'(0)=b\partial_\theta+c\partial_t=\gamma_+'(0)$, the curve $\gamma$ is a $C^1$ and satisfies $\gamma'(t)\in L(\gamma(t))$ for all $t\in (-\varepsilon,\varepsilon)$, as we wanted to prove.
\end{proof} 
By Theorem A in \cite{Bers1956}, if $\sigma$ as in \eqref{formula sigma} is not identically zero, then $w(u,v)=\text{Re}(\alpha z^n)$ for some complex $\alpha\neq 0$ and $n\geq 2$, where $z=u+iv$. The asymptotic lines of $\text{Hess}(w)$ are just a rotation of asymptotic lines generated by $\text{Hess(Re}(z^n))$. When $n\geq 3$, the point $(0,0)$ is a singularity of $\sigma$.

The differential equation of the asymptotic lines of $Hess(w)$ is
\begin{equation*}
    w_{uu}du^2+2w_{uv}dudv-w_{uu}dv^2=0.
\end{equation*}
Writing $z=re^{i\theta}$, we have 
\begin{equation*}
    w_{uu}=n(n-1)r^{n-2}\cos((n-2)\theta),
\end{equation*}
and
\begin{equation*}
    w_{uv}=-n(n-1)r^{n-2}\sin((n-2)\theta), 
\end{equation*}
so the slope of the asymptotic lines depends only on $\theta$:
\begin{equation*}
    \frac{dv}{du}(\theta)=-\tan((n-2)\theta)\pm\sec((n-2)\theta)=:\zeta(\theta).
\end{equation*}

It will be useful to establish a way to compute the angular variation between two vector fields along a curve (see \cite{docarmo2016} Lemma 1, section 4.4). Consider $L(t)$ and $E(t)$ be $C^1$ unit vector fields along a curve $\alpha: [0,1] \to S$. Denote by $\{E(t),E^\perp(t)\}$ an oriented unit orthonormal basis for all $t$. Writing 
\begin{equation*}
    L(t)=a(t)E(t)+b(t)E^\perp(t).
\end{equation*}

 and fixing $\varphi_0$ such that $L(0)=\cos(\varphi_0)E(0)+\sin(\varphi_0)E^\perp(0)$, there exists a $C^1$ function $\varphi: [0,1] \to \mathbb{R}$ such that $L(t)=\cos(\varphi(t))E(t)+\sin(\varphi(t))E^\perp(t)$. Precisely,
\begin{equation}\label{angularvariation}
    \varphi(t)=\varphi_0+\int_0^tF(t)dt
\end{equation}
where $F(t)=a(t)b'(t)-b(t)a'(t)$.

\begin{theorem}\label{thm:negativeindex}
    Let $S$ be the double surface $\widehat{\Sigma}\cup_{id}\widehat{\Sigma}$, and let $L$ be the line field defined in the Proposition \ref{prop:hipoteseshopf}. Then every singularity of $L$ has negative index.
\end{theorem}

\begin{proof}
    Let $p\in S$ be a singular point of $L$. Since around $p$, on both sides of the double $S$, the field $L$ is a small perturbation of the asymptotic lines determined by $Hess(w)$, it suffices to compute the index of this model field, denoted $L'$. 

    Consider a small circle centered in $p$, and let $V$ be a unit vector field tangent to $L'$ along this circle. In coordinates $(u,v)$,
    \begin{equation*}
        V(\theta)=\left( \frac{1}{\sqrt{1+\zeta(\theta)^2}},\frac{\zeta(\theta)}{\sqrt{1+\zeta(\theta)^2}}\right).
    \end{equation*}
    Using \eqref{angularvariation}, the angular variation of $V$ along the arc $\theta_0\leq\theta\leq\theta_1$ is computed with respect to the Euclidean metric $du^2+dv^2$ as follows:
    \begin{align*}
        \delta\measuredangle= & \int_{\theta_0}^{\theta_1}\left[\frac{1}{\sqrt{1+\zeta(\theta)^2}}\left(\frac{\zeta(\theta)}{\sqrt{1+\zeta(\theta)^2}}\right)'-\left(\frac{1}{\sqrt{1+\zeta(\theta)^2}}\right)'\frac{\zeta(\theta)}{\sqrt{1+\zeta(\theta)^2}}\right]d\theta \\
        = & \int_{\theta_0}^{\theta_1} \frac{\zeta'(\theta)}{1+\zeta(\theta)^2}d\theta \\
        = & (n-2)\int_{\theta_0}^{\theta_1}\frac{-\sec^2((n-2)\theta)\pm \sec((n-2)\theta)\tan((n-2)\theta)}{2\sec^2((n-2)\theta)\mp2\sec((n-2)\theta)\tan((n-2)\theta)}d\theta \\
        = & -\frac{n-2}{2}(\theta_1-\theta_0)<0.
    \end{align*}
    Thus, the total angular variation is $-(n-2)\pi$. Since it does not depend on the metric used to compute it, the index of the singularity of $\sigma$ is the negative number $-(n-2)/2$.
\end{proof}

\subsection{Proof of the main theorem}

We now have all the ingredients needed to prove our main result (Theorem \ref{int.theo.1}). For the reader's convenience, we restate it here:

\begin{theorem}\label{main.theorem}
Let $\Omega$ be a smooth domain in a complete orientable Riemannian three-manifold $(M,g)$. Let $\mathcal{A}$ be a class of immersed oriented surfaces in $M$ such that:
    \begin{itemize}
        \item $\mathcal{A}$ is modeled by an elliptic PDE (adapted to $\Omega$) on $M$;
        \item there exists a transitive capillary family $\mathcal{S}\subset\mathcal{A}$ in $M$ with respect to $\Omega$.
    \end{itemize}
 If an immersed connected oriented surface $\Sigma\in \mathcal{A}$ intersects $\Omega$ in an immersed disk satisfying the capillary boundary condition with respect to $\Omega$, then $\Sigma$ is contained in a surface $S\in \mathcal{S}$ that intersects $\partial \Omega$.
\end{theorem}

\begin{proof}
    Assume, by contradiction, that $\Sigma$ is not contained in any surface of the transitive family. Then the symmetric two-tensor $\sigma$ defined in \eqref{int.2} induces asymptotic lines on $\Sigma$ with isolated singularities, possibly on $\partial (\Omega\cap \Sigma)$. 
    Let $L$ be the line field on the double surface $S=\widehat{\Sigma}_1\cup_{id}\widehat{\Sigma}_2$ constructed in Proposition \ref{prop:hipoteseshopf}. By Theorem \ref{thm:negativeindex}, all singularities of $L$ are isolated and have negative index. 
    By the Poincaré-Hopf Index Theorem (see the proof in \cite{Hopf1946}, Chapter III, Theorem 1.6), this implies that the Euler characteristic of $S$ is negative. However, by hypothesis, $\widehat{\Sigma}$ is a disk, and hence $S$ is a sphere, which has positive Euler characteristic. This contradiction proves the theorem.
\end{proof}

\begin{rmk}
    In the same setting, if $\Sigma\in \mathcal{A}$ intersects $\Omega$ as an immersed annulus that satisfies the capillary boundary condition with respect to $\Omega$, then either $\Sigma$ is contained in a member of the transitive family or $\sigma$ has no singularities on $\Sigma\cap \Omega$. Indeed, in this case, the double surface is a torus, which has zero Euler characteristic. Applied to the case of the Euclidean unit ball, with the transitive family of totally geodesic planes, we recover the well-known fact that capillary minimal annuli in the Euclidean ball do not have umbilical points.
\end{rmk}

\section{Applications}

The first application of Theorem \ref{main.theorem} is a uniqueness result for elliptic Weingarten surfaces, analogous to Corollary 4.1 in \cite{Galvez-Mira-Uniqueness}. We consider $M=\mathbb{R}^3$ and $E=\{(x,y,z)\in\mathbb{R}^3;\ z\geq0\}$ and denote by $\kappa_1\geq\kappa_2$ the principal curvatures of a surface in $\mathbb{R}^3$.

\begin{corollary}
    Let $S$ be an immersed sphere in $\mathbb{R}^3$ satisfying the prescribed curvature equation
    \begin{equation}\label{corollary1eq1}
    \Phi(\kappa_1,\kappa_2,\eta)=0
    \end{equation}
    
    where, for each fixed $\eta_0\in \mathbb{S}^2$, the function $\Phi(x,y,\eta_0)$ is $C^{1,\alpha}$, symmetric in $x$ and $y$, and satisfies the ellipticity condition
    
    \begin{equation}\label{corollary1eq2}
    \frac{\partial\Phi}{\partial x}\frac{\partial\Phi}{\partial y}>0.
    \end{equation}
    
    Assume that $S$ is rotationally symmetric about the $z$-axis and that its Gauss map is a diffeomorphism. If $\Sigma$ is a connected oriented surface satisfying \eqref{corollary1eq1} and intersecting the half-space $E$ in a capillary disk, then $\Sigma$ is contained in a translation of $S$ that intersects $\partial E$.
\end{corollary}


\begin{proof}
    Since $S$ is rotationally symmetric about the $z$-axis, $\partial E=\{z=0\}$, and the Gauss map of $S$ is a diffeomorphism, the family of all translations of $S$ constitutes a transitive capillary family in $\mathbb{R}^3$ with respect to $E$. All such surfaces satisfy \eqref{corollary1eq1}. 
    
    The symmetry of $\Phi$ in $x$ and $y$, together with \eqref{corollary1eq2}, implies that $\Phi$ satisfies an ellipticity condition as presented in the introduction of this article (see \cite{Alexandrov1962a}).
    
    Moreover, if a translation $S'$ of $S$ intersects $\partial E$ transversally, then $S'\cap E$ is a capillary disk in $E$. Therefore, all hypotheses of Theorem \ref{main.theorem} are verified.
\end{proof}

The main theorem can also be used to derive non-existence results. Let $B_r(p)$ denote the open geodesic ball in the Euclidean sphere $\mathbb{S}^3$ with radius $r\in (0,\pi/2)$ and center at $p$.

\begin{theorem}
    There are no minimal capillary disks in the domain $\Omega=\mathbb{S}^3\setminus (B_r(p)\cup B_r(-p))$.
\end{theorem}

\begin{proof}
    The intersections of equators of $\mathbb{S}^3$ with $\Omega$ form a transitive capillary family with respect to $\Omega$. The intersection of each member of this transitive family with $\Omega$ is either a sphere or annuli; in particular, it contains no disks.
    
    Let $\Sigma$ be an immersed capillary minimal surface in $\Omega$. Then $\Sigma$ extends to a minimal surface in $\mathbb{S}^3$ (we may again apply Theorem 5.2 of \cite{KinNirSpr} to conclude that the surface is analytic up to the boundary, and hence admits a local extension beyond $\Omega$ as a minimal surface). 
    
    If, in addition, $\Sigma$ is a disk, we may apply Theorem \ref{main.theorem} to its extension and conclude that $\Sigma$ is given by the transversal intersection of $\Omega$ with a member of the above capillary transitive family relative to $\Omega$. This is a contradiction, since such intersections are never disks.
\end{proof}




\printbibliography[heading=bibintoc]
\end{document}